\documentclass[12pt,leqno]{amsart}
\usepackage{amssymb,amsthm,amsmath,latexsym,mathtools}
\usepackage[all]{xy}
\usepackage{lineno}
\newcommand{\vfi}{\varphi}

\newcommand{\p}{{\bf p}}
\newtheorem{theorem}{\sc Theorem}[section]
\newtheorem{thm}[theorem]{\sc Theorem}
\newtheorem{lem}[theorem]{\sc Lemma}
\newtheorem{prop}[theorem]{\sc Proposition}
\newtheorem{cor}[theorem]{\sc Corollary}

\newtheorem{rem}[theorem]{\sc Remark}

\title[non-abelian Tensor Square of $p$-groups]{The exponent of the non-abelian tensor square and related constructions of $p$-groups}
\author[Bastos]{R.  Bastos}
\author[de Melo]{E. de Melo}
\author[Gon\c calves]{N. Gon\c calves}
\address{Departamento de Matem\'atica, Universidade de Bras\'ilia,
Brasilia-DF, 70910-900 Brazil}
\email{(Bastos) bastos@mat.unb.br, (de Melo) emerson@mat.unb.br and (Gon\c calves) nathalianogueira208@yahoo.com.br}
\author[Monetta]{C. Monetta }
\address{Dipartimento di Matematica, Universit\`a di Salerno, Via Giovanni Paolo II, 132 - 84084 - Fisciano (SA), Italy}
\email{cmonetta@unisa.it}
\subjclass[2010]{20D15, 20E06, 20J06}
\keywords{Finite $p$-groups; Non-abelian tensor square of groups}

\usepackage{graphicx}

\begin{document}

\maketitle

\begin{abstract}
Let $G$ be a finite $p$-group. In this paper we obtain bounds for the exponent of the non-abelian tensor square $G \otimes G$ and of $\nu(G)$, which is a certain extension of $G \otimes G$ by $G \times G$. In particular, we bound $\exp(\nu(G))$ in terms of $\exp(\nu(G/N))$ and $\exp(N)$ when $G$ admits some specific normal subgroup $N$.
We also establish bounds for $\exp(G \otimes G)$ in terms of $\exp(G)$ and either the nilpotency class or the coclass of the group $G$, improving some existing bounds. 
\end{abstract}

\section{Introduction}

The non-abelian tensor square $G \otimes G$ of a group $G$ was introduced by Brown and Loday \cite{BL} following works of Miller \cite{Miller} and Dennis \cite{Dennis}. It is defined to be the group generated by all symbols $\; \, g\otimes h, \; g,h\in G$, subject to the relations
\[
gg_1 \otimes h = ( g^{g_1}\otimes h^{g_1}) (g_1\otimes h) \quad
\mbox{and} \quad g\otimes hh_1 = (g\otimes h_1)( g^{h_1} \otimes
h^{h_1})
\]
for all $g,g_1, h,h_1 \in G$, where  we write $x^y$ for the conjugate $y^{-1} x y$ of $x$ by $y$, for any elements $x, y \in G$. In \cite{BL}, Brown and Loday showed that the third homotopy group of the suspension of an Eilenberg-MacLane space $K(G,1)$ satisfies $\pi_3(SK(G,1)) \cong \mu(G),$ where $\mu(G)$ denotes the kernel of the derived map $\rho': G \otimes G \to G'$, given by  $g \otimes h \mapsto [g,h]$. According to \cite[Proposition 2.8]{NR2}, the sequence
$$
1 \rightarrow \Delta(G)\rightarrow \mu(G) \rightarrow H_2(G)\rightarrow 1,
$$
is exact, where $\Delta(G) = \langle g \otimes g \mid g \in G\rangle$ and $H_2(G)$ is the second homology group of the group $G$. When $G$ is finite, the Schur multiplier of $G$, denoted by $M(G)$, is defined to be $M(G)=H_2(G)$. Here $\rho'$  corresponds to the derived map $\kappa$ of \cite{BL}. 

We observe that the defining relations of the non-abelian tensor square can be viewed as abstractions of commutator relations; thus in \cite{NR1}, Rocco considered the following construction (see also Ellis and Leonard \cite{EL}). Let $G$ be a group and let $\varphi : G \rightarrow G^{\varphi}$ be an isomorphism ($G^{\varphi}$ is a copy of $G$, where $g \mapsto g^{\varphi}$, for all $g \in G$). Define the group $\nu(G)$ to be \[ \nu (G):= \langle 
G \cup G^{\varphi} \ \vert \ [g_1,{g_2}^{\varphi}]^{g_3}=[{g_1}^{g_3},({g_2}^{g_3})^{\varphi}]=[g_1,{g_2}^{\varphi}]^{{g_3}^{\varphi}},
\; \ g_i \in G \rangle .\]

The group $\nu(G)$ can be viewed as a special semi-direct product $\nu(G) \cong ((G \otimes G)\rtimes G) \rtimes G$ (see \cite[Section 2]{EL} for more details). The motivation for studying $\nu(G)$ is the commutator connection: indeed, the map  $\Phi: G \otimes G \rightarrow [G, G^{\varphi}]$,
defined by $g \otimes h \mapsto [g , h^{\varphi}]$, for all $g, h \in G$, is an isomorphism  \cite[Proposition 2.6]{NR1}. From now on we identify the non-abelian tensor square $G \otimes G$ with the subgroup $[G,G^{\varphi}]$ of $\nu(G)$. The group $\nu(G)$ provides an interesting computational tool in the study of the non-abelian tensor square of groups (see for instance \cite{BdMGN,BFM,BN08, EL,M09,NR2}). 

Our purpose is to achieve bounds for the exponent of the non-abelian tensor square and related constructions of finite $p$-groups. It is worth to mention that the bounds obtained for the exponent of the group $\nu(G)$ can be read as bounds for the exponents of its sections, and the other way around. Therefore, for the sake of completeness we summarize the relationship between the exponent of $\nu(G)$ and its sections in Remark \ref{rem:nu(G)}. 

Let $p$ be a prime. A finite $p$-group $G$ is said to be {\em powerful} if $p>2$ and $G' \leq G^p$, or $p=2$ and $G' \leq G^4$. A more general class of $p$-groups is the following. We say that a finite $p$-group is {\em potent} if $p>2$ and $\gamma_{p-1}(G) \leq G^p$, or $p=2$ and $G' \leq G^4$. Note that the family of potent $p$-groups contains all powerful $p$-groups. Recall that a subgroup $N$ of $G$ is potently embedded in $G$ if $[N,_{p-2}G]\leq N^p$ for $p$ odd, or $[N, G]\leq N^4$ for $p=2$ ($N$ is powerfully embedded in $G$ if $[N,G]\leq N^p$ for $p$ odd, or $[N, G]\leq N^4$ for $p=2$). More information about finite powerful and potent  $p$-groups can be found in \cite{D} and in \cite{JJ}, respectively. 

Let $p$ be a prime and $r$ a positive integer. We define the integer $m=m(p,r)$ by $m(p,r)=(p-1)p^{r-1}$ for $p$ odd and $m(2,r)=2^{r+2}$.  Recall that the coclass of a finite $p$-group $G$ of order $p^n$ and nilpotency class $c$ is defined to be $r(G)=n-c$. Let $G$ be a $p$-group of coclass $r=r(G)$ and nilpotency class $c$, where $c\geq 2p^r$ if $p$ is odd or $c\geq 2^{r+3}$ if $p=2$. It is well-known that in this case $\gamma_{i+s}(G)=\gamma_i(G)^p$ for all $i\geq m(p,r)$ and $s=(p-1)p^{d}$ with $0\leq d \leq r-1$ if $p$ is odd or $s=2^d$ with $0\leq d \leq r+1$ if $p=2$ (cf. \cite[Section  6.3]{LM}). In particular, note that $\gamma_m(G)$ is powerful. It is worth to mention that powerful $p$-groups satisfy some analogous power-commutator condition. Actually, if $G$ is a powerful $p$-group, then $\gamma_i(G)$ is powerfully embedded in $G$, that is, $\gamma_{i+1}(G) \leq \gamma_i(G)^{\p}$, for every $i\geq 1$. (cf. \cite[Corollary 11.6]{Khukhro}). Here and in the sequel, $\p$ denotes the prime $p$ if $p$ is odd and $4$ if $p=2$.

In \cite{M09}, Moravec proved that $[G,G^{\varphi}]$ and $\nu(G)'$ are powerfully embedded in $\nu(G)$; moreover, the exponent $\exp(\nu(G)')$ divides $\exp(G)$. Later, in \cite{BdMGN} it was shown that if $G$ is a finite potent $p$-group, then the subgroups $\gamma_k(\nu(G))$ and $[G,G^{\varphi}]$ are potently embedded in $\nu(G)$; and $\exp(\nu(G))$ divides $\p \cdot \exp(G)$. Furthermore, if $p \geq 5$ and $G$ is a powerful $p$-group, then $\exp(\nu(G)) = \exp(G)$. In some sense the next result can be viewed as an extension of these results and also as an application for $p$-groups of coclass $r$.   

\begin{thm}\label{thmA}
Let $p$ be a prime and $G$ a $p$-group. Let $m$ and $s$ be positive integers such that $m\geq s$ and   suppose that $\gamma_{i+s}(G)=\gamma_i(G)^p$ for every $i \geq m$. Then
\begin{itemize}
    \item[(a)] $\gamma_{i+s+1}(\nu(G))=\gamma_{i+1}(\nu(G))^p$ for $i> m$;
    \item[(b)] if $p$ is odd, then  $\exp(\gamma_{m+1}(\nu(G)))$ divides $\exp(\gamma_{m}(G))$;
    \item[(c)] if $p=2$ and $\gamma_m(G)$ is powerful, then $\exp(\gamma_{m+1}(\nu(G)))$ divides $\exp(\gamma_{m}(G)).$
\end{itemize}
\end{thm}

Despite the fact that the coclass of $\nu(G)$ grows faster than the nilpotency class of the involved group $G$ (see Remark \ref{prop.coclass}, below), Theorem \ref{thmA} (a) shows that the group $\nu(G)$ still satisfies a close power-commutator condition, as $G$ does. At the same time, according to Theorem \ref{thmA} (b) and (c), we deduce that the behaviour of the exponent $\exp(\gamma_{m+1}(\nu(G)))$ depends only on $\exp(\gamma_m(G))$.

Later, we obtain bounds for the exponent of $\nu(G)$ in terms of some specific normal subgroups of $G$.  

\begin{thm}\label{thm.potent}
Let $p$ be a prime and $N$ a normal subgroup of a $p$-group $G$. 
\begin{itemize}
    \item[(a)] If $N$ is potent or $\gamma_{p}(N)=1$, then $\exp(\nu(G))$ divides $\p\cdot \exp(\nu(G/N))\cdot\exp(N)$.
    \item[(b)] If $\gamma_{p-2}(N) \leq N^p$,  then $\exp(\nu(G))$ divides $\exp(\nu(G/N))\cdot\exp(N)$. 
\end{itemize}
\end{thm}

In \cite{BJR}, Brown, Johnson and Robertson described the non-abelian tensor square of $2$-groups of maximal class (i.e., groups of coclass $1$). In particular, if $G$ is a $2$-group of maximal class, then $\exp([G,G^{\varphi}])$ divides $\exp(G)$ (cf. \cite[Propositions 13--15]{BJR}). Consequently, $\exp(\nu(G))$ divides $\exp(G)^2$. In \cite{Moravec.cc}, Moravec proved that if $G$ is a $p$-group of maximal class, then $\exp(M(G))$ divides $\exp(G)$.

\begin{cor} \label{cor.maximal}
Let $p$ be a prime and $G$ a $p$-group of maximal class. Then $\exp(\nu(G))$ divides $\p^2\cdot \exp(G)$.
\end{cor}

In the literature, the exponent of several sections of the group $\nu(G)$, like $G \otimes G$, $\mu(G)$ and $M(G)$, has been investigated (see \cite{Sambonet} and the references given there). In \cite{Ellis}, Ellis proved that if $G$ is a $p$-group of class $c\geq 2$, then $\exp([G,G^{\varphi}])$ divides $\exp(G)^{c-1}$. In \cite{Moravec.Schur}, Moravec showed that if $G$ is a $p$-group of class $c\geq 2$, then $\exp(M(G))$ divides $\exp(G)^{2\lfloor \log_2(c)\rfloor}$. Later, in \cite{Sambonet17}, Sambonet proved that if $G$ is a $p$-group of class $c\geq 2$, then $\exp(M(G))$ divides $\exp(G)^{\lfloor \log_{p-1}(c)\rfloor+1}$ if $p>2$ and $\exp(M(G))$ divides $2^{\lfloor \log_2(c)\rfloor} \cdot \exp(G)^{\lfloor \log_2(c)\rfloor+1}$ if $p=2$. In \cite{APT}, Antony et al. demonstrated that $\exp(M(G))$ divides $\exp(G)^{1 + \lceil \log_{p-1}(\frac{c+1}{p+1})\rceil}$ if $p \leq c$, improving all the previous bounds.

Our contribution is a bound for the $\exp([G,G^{\varphi}])$ which, in the realm of Remark \ref{rem:nu(G)}, improves the bound obtained in \cite{APT}.

\begin{thm}\label{corlog}
Let $p$ be a prime and $G$ a $p$-group of nilpotency class $c$. Let $n=\lceil \log_{p}(c+1)\rceil$. Then $\exp([G,G^{\varphi}])$ divides $\exp{(G)}^n$.
\end{thm}

Furthermore, in \cite{Moravec.cc} Moravec proved that if $G$ is a $p$-group of coclass $r$, then $\exp(M(G))$ and $\exp(G \wedge G)$ divide $\exp(G)^{r+1+2 \left \lfloor{\log_2(m-1)}\right \rfloor}$, where $m=m(p,r)$ is as defined before. Finally, we obtain the following bounds for the exponent of $[G,G^{\varphi}]$.

\begin{thm} \label{thm.cc}
Let $p$ be a prime and $G$ a $p$-group of coclass $r$.
\begin{itemize}
    \item[(a)] If $p$ is odd, then  $\exp([G,G^{\varphi}])$ divides $(\exp(G))^{r}\cdot \exp(\gamma_m(G))$, where $m=(p-1)p^{r-1}$.
    \item[(b)] If $p=2$, then $\exp([G,G^{\varphi}])$ divides $(\exp(G))^{r+3}\cdot \exp(\gamma_m(G))$, where $m=2^{r+2}$.
\end{itemize}
\end{thm}

We point out that both previous results implies a bound for $\exp(\nu(G))$ (cf. Remark \ref{rem:nu(G)}). 

As a consequence of Theorem~\ref{thm.cc} we obtain the following.

\begin{cor}\label{cor.explicit.cc}
Let $p$ be a prime and $G$ a $p$-group of coclass $r$. 
\begin{itemize}
   \item[(a)] If $p\geq 3$ then $\exp(M(G))$ and $\exp(\mu(G))$ divide $ \exp(G)^{r+1}$;
    \item[(b)] If $p=2$ then $\exp(M(G))$ and $\exp(\mu(G))$ divide $ \exp(G)^{r+3}$.
\end{itemize}
\end{cor}

It is worth to mention that for every prime $p$ the bounds obtained in Corollary~\ref{cor.explicit.cc} improve the ones obtained in \cite{Moravec.cc} when the coclass $r$ is at least $2$ (cf. \cite[Corollary 4.8]{Moravec.cc}). Furthermore, in the context of Sambonet theorem \cite[Theorem 3.3]{Sambonet17}, the improvement occurs for $e\leq r$ if $p>2$ and $e\leq r+2$ if $p=2$, where $\exp(G)=p^e$. \\[2mm]

The paper is organized as follows. In Section 2 we collect results of general nature that are later used in the proofs of our main theorems. The third section is devoted to the proof of Theorem \ref{thmA}. The proofs of Theorems \ref{thm.potent} and \ref{corlog} are given in Section 4. We also obtain bounds for the exponent $\exp([G,G^{\varphi}])$ in terms of some potent normal subgroups (see Corollary \ref{cor.potent}, below). In Section 5 we prove Corollary \ref{cor.maximal} and Theorem \ref{thm.cc}. 


\section{Preliminaries}

\subsection{Finite $p$-groups}
In this subsection we summarize without proofs the relevant material on finite $p$-groups. 

\begin{lem}(\cite[Lemma 2.2]{JJJ}) \label{normalinc}
    Let $G$ be a finite $p$-group and $N$, $M$ normal subgroups of $G$. If $N\leq M[N,G]N^p$ then $N \leq M$.
\end{lem}

The following theorem is known as P. Hall's collection formula.

\begin{thm}(\cite[Appendix A]{D}) \label{thm.Hall}
Let $G$ be a $p$-group and $x, y$ elements of $G$. Then for any $k\geq 0$ we have  
\[
(xy)^{p^k}\equiv x^{p^k}y^{p^k} \pmod{\gamma_{2}(L)^{p^k}\gamma_{p}(L)^{p^{k-1}}\gamma_{p^2}(L)^{p^{k-2}}\gamma_{p^3}(L)^{p^{k-3}}\cdots \gamma_{p^k}(L)},
\]
where $L=\langle x,y\rangle$. We also have 
\[
[x,y]^{p^k}\equiv [x^{p^k}, y] \pmod {\gamma_{2}(L)^{p^k}\gamma_{p}(L)^{p^{k-1}}\gamma_{p^2}(L)^{p^{k-2}}\ldots \gamma_{p^k}(L)},
\]
where $L=\langle x,[x,y]\rangle$.
\end{thm}

\begin{cor}(\cite[Theorem 2.3]{JJJ}) \label{cor.Hall}
    Let $G$ be a $p$-group and $x_1,\ldots, x_r$ elements of $G$. Then for any $k\geq 0$ we have
    \[
    (x_1\ldots x_r)^{p^k}\equiv x_1^{p^k}\ldots x_r^{p^k} \pmod{\gamma_{2}(L)^{p^k}\gamma_{p}(L)^{p^{k-1}}\gamma_{p^2}(L)^{p^{k-2}}\gamma_{p^3}(L)^{p^{k-3}}\cdots \gamma_{p^k}(L)},
    \]
    where $L=\langle x_1,\ldots, x_r\rangle$.
\end{cor}

A consequence of P. Hall's collection formula is given by the following lemma
\begin{lem}(\cite[Theorem 2.4]{JJJ}) \label{lem.hallformula}
Let $G$ be a finite $p$-group and $N$, $M$ normal subgroups of $G$. Then $$[N^{p^k},M]\equiv [N,M]^{p^{k}}(\text{mod}\ [M,_pN]^{p^{k-1}} [M,_{p^2}N]^{p^{k-2}} \ldots [M,_{p^k}N]).$$
\end{lem}

\begin{lem}[\cite{JJ}]\label{lem.lower.potent}
Let $G$ be a potent $p$-group and $k\geq 1$. If $p=2$, then $\gamma_{k+1}(G)\leq\gamma_k(G)^4$, and if $p\geq 3$ then $\gamma_{p-1+k}(G)\leq\gamma_{k+1}(G)^p$.
\end{lem}

The next lemma will be useful to determine the exponent of the group $\nu(G)$ in terms of $\exp(G)$ for some $p$-groups. 

\begin{lem}[\cite{JJJ}] \label{lem.exponent}
Let $G$ be a finite $p$-group and $k\geq 1$. Assume that $\gamma_{k(p-1)}(G) \leq \gamma_r(G)^{p^s}$ for some $r$ and $s$ such that $k(p-1) < r + s(p-1)$. Then the exponent $\exp(\Omega_i(G))$ is at most $p^{i+k-1}$ for all $i$. 
\end{lem}

Let $G$ be a $p$-group. We define $\Pi_i(G)$  inductively by: $\Pi_0(G)=G$ and $\Pi_i(G)=(\Pi_{i-1}(G))^p$ for $i>0$. The next result will be needed in the proof of Theorem \ref{thmA}.

\begin{lem}[\cite{D}]\label{power}
Let $G$ be a powerful $p$-group and $i\geq 1$. Then $\Pi_i(G) = G^{p^{i}}$ for every $i\geq 1$.
\end{lem}

\subsection{The group $\nu(G)$}

This subsection will be devoted to describe some properties of the group $\nu(G)$. 

The following basic properties are consequences of 
the defining relations of $\nu(G)$ and the commutator rules (see \cite[Section 2]{NR1} and \cite[Lemma 1.1]{BFM} for more details). 

\begin{lem} 
\label{basic.nu}
The following relations hold in $\nu(G)$, for all 
$g, h, x, y \in G$.
\begin{itemize}
\item[$(i)$] $[g, h^{\varphi}]^{[x, y^{\varphi}]} = [g, h^{\varphi}]^{[x, 
y]}$; 
\item[$(ii)$] $[g, h^{\varphi}, x^{\varphi}] = [g, h, x^{\varphi}] = [g, 
h^{\varphi}, x] = [g^{\vfi}, h, x^{\vfi}] = [g^{\vfi}, h^{\vfi}, x] = 
[g^{\vfi}, 
h, x]$;
\item[$(iii)$] $[[g,h^{\varphi}],[x,y^{\varphi}]] = [[g,h],[x,y]^{\varphi}]$.
\end{itemize}
\end{lem}

Let $N$ be a normal subgroup of a finite group $G$. We denote by $K$ the subgroup $[N,G^{\varphi}] [G,N^{\varphi}] \cdot \langle N,N^{\varphi}\rangle$ in $\nu(G)$, where the dot means internal semidirect product. We set $\overline{G}$ for the quotient group $G/N$ and the canonical epimorphism $\pi: G \to \overline{G}$ gives rise to an epimorphism $\widetilde{\pi}: \nu(G) \to \nu(\overline{G})$ such that $g \mapsto \overline{g}$, $g^{\varphi} \mapsto \overline{g^{\varphi}}$, where $\overline{G^{\varphi}} = G^{\varphi}/N^{\varphi}$ is identified with $\overline{G}^{\varphi}$. 

\begin{lem}(Rocco, \cite[Proposition 2.5 and  Remark 3]{NR1})\label{lem.general} With the above notation we have 

\begin{itemize}
 \item[$(a)$] $[N,G^{\varphi}] \unlhd \nu(G)$, $[G,N^{\varphi}] \unlhd \nu(G)$;
 \item[$(b)$] $\ker(\widetilde{\pi}) = [N,G^{\varphi}] [G,N^{\varphi}] \cdot \langle N,N^{\varphi}\rangle = ([N,G^{\varphi}] [G,N^{\varphi}]\cdot  N) \cdot N^{\varphi}$.
 \item[$(c)$] There is an exact  sequence
\[
1 \rightarrow [N,G^{\varphi}][G,N^{\varphi}] \rightarrow{} [G,G^{\varphi}] \rightarrow   \left[ G/N,\left( G/N\right)^{\varphi}\right] \rightarrow 1.
\] 
 \end{itemize}
\end{lem}

We need the following description to the lower central series of $\nu(G)$. 

\begin{prop}\cite[Proposition 2.7]{BuenoRocco}\label{gammanu}
Let $k$ be a positive integer and $G$ a group. Then
$\gamma_{k+1}(\nu(G)) = \gamma_{k+1}(G)\gamma_{k+1}(G^{\varphi})[\gamma_{k}(G), G^{\varphi}]$.
\end{prop}

The above result shows that if $G$ is nilpotent of class $c$, then the group $\nu(G)$ is nilpotent of class at most $c+1$. On the other hand, the coclass of the group $\nu(G)$ has a different behaviour.

\begin{rem} \label{prop.coclass}
Let $G$ be a finite $p$-group. Assume that $G$ has coclass $r$ and order $|G| = p^n$.  Then the coclass $r(\nu(G))$ is at least $r+2n-1$.
\end{rem}

\begin{proof}
First we prove that $|G| \leq |[G,G^{\varphi}]|$. Since $[G,G^{\varphi}]/\mu(G)$ is isomorphic to $G'$, it suffices to show that the order of the abelianization $|G^{ab}|$ divides $|\mu(G)|$. Indeed, by \cite[Remark 5]{NR1}, $|G^{ab}| \leq |\Delta(G)|$, where $\Delta(G) = \langle [g,g^{\varphi}] \mid g \in G\rangle \leq \mu(G)$. From this we deduce that $|\nu(G)| = p^{\alpha}\geq p^{3n}$. By Proposition \ref{gammanu}, the nilpotency class $c(\nu(G))$ of $\nu(G)$ is at most $c+1$, where $c=n-r$. Consequently,  $$r(\nu(G)) =  \alpha - c(\nu(G)) \ \geq \ 3n - (c+1) \ = \ r + 2n-1,  $$
which establishes the formula.
\end{proof}

As $\nu(G)$ is an extension of $[G,G^{\varphi}]$ by $G \times G$, we have $\exp(\nu(G))$ divides $\exp(G) \cdot \exp([G,G^{\varphi}])$.
Combining \cite{BFM} and \cite{NR2}, we deduce the following bounds for $\exp(\nu(G))$ in terms of $\exp(\mu(G))$, $\exp(M(G))$.  

\begin{rem} \label{rem:nu(G)}
Consider the following exact sequences  (Rocco, \cite{NR2}),
$$
 1 \rightarrow [G,G^{\varphi}] \rightarrow \nu(G) \rightarrow G \times G \rightarrow 1, 
$$

$$
1 \rightarrow \Theta(G) \rightarrow \nu(G) \rightarrow G \rightarrow 1
$$
and 
$$
 1 \rightarrow \Delta(G) \rightarrow \mu(G) \rightarrow M(G) \rightarrow 1,
$$
where $\Theta(G)$ is the kernel of the epimorphism $\rho: \nu(G) \to G$, given by $g \mapsto g$ and $g^{\varphi} \mapsto g$. By \cite[Section 2]{NR2}, $\mu(G) = \Theta(G) \cap [G,G^{\varphi}]$. 

Let $G$ be a finite group. By the first and second exact sequence, we deduce that $\nu(G)/\mu(G)$ is isomorphic to a subgroup of $G \times G \times G$ and so, $\exp(\nu(G))$ divides $\exp(G) \cdot \exp(\mu(G))$. Moreover, by the third exact sequence, we conclude that $\exp(\mu(G))$ divides $\exp(M(G)) \cdot \exp(\Delta(G))$. Moreover, as $[g^j,g^{\varphi}] = [g,g^{\varphi}]^j$ for any $g \in G$, we have $\exp(\Delta(G))$ divides $\exp(G)$. Consequently, $$\exp(\nu(G)) \ \text{divides} \  \exp(G)
^2 \cdot \exp(M(G)).$$ 
Assume that $2$ does not divide $|G^{ab}|$, where $G^{ab}=G/G'$. According to \cite[Corollary 1.4]{BFM}, we deduce that $\mu(G) \cong M(G) \times \Delta(G)$ and so, $$\exp(\nu(G)) \ \text{divides} \ \exp(G) \cdot \max \{\exp(G), \exp(M(G))\}.$$  
\end{rem}


\section{Power-commutator conditions and the exponent of the lower central terms of $\nu(G)$}

Under the hypothesis of Theorem~\ref{thmA} we will prove the following proposition.
\begin{prop}\label{prop.conditions}
\begin{itemize}
\ 
    \item[(1)] For every odd prime $p$ or $i>m$ we have $\gamma_{i+s+1}(\nu(G)) \leq \gamma_{i+1}(\nu(G))^p$.
    \item[(2)] For every prime $p$ and $i\geq m$ we have $\gamma_{i+1}(\nu(G))^p \leq \gamma_{i+s+1}(\nu(G))$.
\end{itemize}
\end{prop}

\begin{proof}
 (1) We start proving that $\gamma_{i+s+1}(\nu(G)) \leq \gamma_{i+1}(\nu(G))^p.$ From Proposition~\ref{gammanu} and by hypothesis we have
\begin{align*}
\gamma_{i+s+1}(\nu(G)) &= \gamma_{i+s+1}(G)\  \gamma_{i+s+1}(G^{\varphi})\ [\gamma_{i+s}(G), G^{\varphi}]\\
&= \gamma_{i+1}(G)^p \ \gamma_{i+1}(G^{\varphi})^p \ [\gamma_i(G)^p, G^{\varphi}].
\end{align*}

Since both $\gamma_{i+1}(G)^p$ and $\gamma_{i+1}(G^{\varphi})^p$ are contained in $\gamma_{i+1}(\nu(G))^p$, it suffices to prove that $[\gamma_{i}(G)^p, G^{\varphi}] \leq \gamma_{i+1}(\nu(G))^p$. For, let $x \in \gamma_i(G)$ and $y^{\varphi} \in G^{\varphi}$. Then, applying Theorem~\ref{thm.Hall} we have 
\[
[x^p, y^{\varphi}] \equiv [x,y^{\varphi}]^p  \pmod{\gamma_2(L)^p \ \gamma_p(L)},
\]
where $L= \langle x, [x,y^{\varphi}]\rangle$. On the one hand, 
\[
\gamma_2(L)^p \leq [\gamma_i(G), G^{\varphi}, \gamma_i(G)]^p \leq \gamma_{2i+1}(\nu(G))^p.
\]

For every prime $p$, if $i>m$
\begin{align*}
 \gamma_2(L) &\leq [\gamma_i(G), G^{\varphi}, \gamma_i(G)] \leq \gamma_{i+s+2}(\nu(G))\\
&=\gamma_{i+s+2}(G)\gamma_{i+s+2}(G^{\varphi})[\gamma_{i+s+1}(G),G^{\varphi}]\\
&\leq \gamma_{i+1}(\nu(G))^p[\gamma_{i}(G)^p,G,G^{\varphi}]  .   
\end{align*}

On the other hand, $p \geq 3$ implies $2i+p-3 \geq i+s$ and we have
\begin{align*}
\gamma_p(L) & \leq [\gamma_i(G), G^{\varphi}, \gamma_i(G),_{p-2} \nu(G)] = [\gamma_{i+1}(G), \gamma_i(G^{\varphi}),_{p-2} \nu(G)]\\[2mm]
& \leq [\gamma_{2i+1}(G),_{p-3} \nu(G), G^{\varphi}] \leq [\gamma_{2i+p-3}(G), G^{\varphi}, G^{\varphi}] \\[2mm]
& \leq [\gamma_{2i+p-3}(G), G^{\varphi}, \nu(G)] \leq [\gamma_{s+i}(G), G^{\varphi}, \nu(G)] = [\gamma_i(G)^p, G^{\varphi}, \nu(G)].
\end{align*}
Therefore, it follows that \[ [x^p, y^{\varphi}] \in \gamma_{i+1}(\nu(G))^p  [\gamma_i(G)^p, G^{\varphi}, \nu(G)] \] which yields
\[
[\gamma_{i}(G)^p, G^{\varphi}] \leq	\gamma_{i+1}(\nu(G))^p  [\gamma_i(G)^p, G^{\varphi}, \nu(G)].
\]
Applying Lemma~\ref{normalinc} with $N=[\gamma_{i}(G)^p, G^{\varphi}]$ and $M=\gamma_{i+1}(\nu(G))^p$, we can conclude that  $[\gamma_{i}(G)^p, G^{\varphi}] \leq \gamma_{i+1}(\nu(G))^p$.\\
(2) In order to prove that $\gamma_{i+1}(\nu(G))^p \leq \gamma_{i+s+1}(\nu(G))$, consider the subgroup $W=\gamma_{i+1}(G)^p \ \gamma_{i+1}(G^{\varphi})^p \ [\gamma_i(G), G^{\varphi}]^p$. Firstly, we show that
\[
W \equiv \gamma_{i+1}(\nu(G))^p \pmod{\gamma_{i+s+1}(\nu(G))}.
\]
By definition, $W \leq \gamma_{i+1}(\nu(G))^p\leq \gamma_{i+1}(\nu(G))^p\gamma_{i+s+1}(\nu(G))$, so we only need to prove that $\gamma_{i+1}(\nu(G))^p \leq W \gamma_{i+s+1}(\nu(G))$. For, let $\alpha \in \gamma_{i+1}(G)$, $\beta \in \gamma_{i+1}(G^{\varphi})$ and $\delta \in [\gamma_i(G), G^{\varphi}]$. Then, applying Corollary~\ref{cor.Hall} we have
\[
(\alpha \beta \delta)^p \equiv \alpha^p \beta^p \delta^p \pmod{ \gamma_2(J)^p \gamma_p(J)},
\]
where $J=\langle \alpha, \beta, \delta \rangle$. It is straightforward to see that $\alpha^p \beta^p \delta^p \in W$. Moreover, all the generators of $\gamma_2(J)$ belong to $\gamma_{i+s+1}(\nu(G))$. Indeed, 
\begin{align*}
&[\alpha, \beta] \in [\gamma_{i+1}(G), \gamma_{i+1}(G^{\varphi})] \leq [\gamma_{i+1}(\nu(G)), \gamma_{i+1}(\nu(G))] \leq \gamma_{i+s+1}(\nu(G));\\[2mm]
&[\delta, \alpha] \in [\gamma_{i}(G), G^{\varphi}, \gamma_{i+1}(G)] \leq [\gamma_{i+1}(\nu(G)), \gamma_{i+1}(\nu(G)] \leq \gamma_{i+s+1}(\nu(G));\\[2mm]
&[\delta, \beta] \in [\gamma_{i}(G), G^{\varphi}, \gamma_{i+1}(G^{\varphi})] \leq [\gamma_{i+1}(\nu(G)), \gamma_{i+1}(\nu(G)] \leq \gamma_{i+s+1}(\nu(G)).
\end{align*}
Therefore, $\gamma_2(J)^p \leq \gamma_2(J) \leq \gamma_{i+s+1}(\nu(G))$. Furthermore, $\gamma_p(J) \leq  \gamma_2(J) \leq \gamma_{i+s+1}(\nu(G)).$
It follows that $(\alpha \beta \delta)^p \in W \gamma_{i+s+1}(\nu(G))$, and so
\[
\gamma_{i+1}(\nu(G))^p \leq W \gamma_{i+s+1}(\nu(G)).
\]

To conclude, we prove that $W \leq \gamma_{i+s+1}(\nu(G))$, that is, $[\gamma_i(G), G^{\varphi}]^p \leq \gamma_{i+s+1}(\nu(G))$. Let $\alpha=\alpha_1^p\ldots\alpha_n^p\in [\gamma_i(G), G^{\varphi}]^p$, where each $\alpha_j\in [\gamma_i(G), G^{\varphi}]$. We can write $\alpha_j=[x_{j1}, y_{j1}^{\varphi}]\ldots [x_{jl}, y_{jl}^{\varphi}]$, with $x_{jk}\in \gamma_i(G)$ and $y_{jk}^\varphi\in G^\varphi$, for all $k\in\{1, \ldots, l\}$, where $l$ depends on $j$. Applying Corollary~\ref{cor.Hall}
\[
([x_{j1}, y_{j1}^{\varphi}]\ldots [x_{jl}, y_{jl}^{\varphi}])^p \equiv [x_{j1}, y_{j1}^{\varphi}]^p\ldots [x_{jl}, y_{jl}^{\varphi}]^p \pmod{ \gamma_2(S)^p \gamma_p(S)},
\]
where $S=\langle [x_{j1}, y_{j1}^{\varphi}], \ldots, [x_{jl}, y_{jl}^{\varphi}] \rangle$. Observe that
\begin{align*}
&\gamma_2(S)^p \gamma_p(S) \leq \gamma_2(S) \leq [\gamma_i(G), G^{\varphi}, [\gamma_i(G), G^{\varphi}]] \leq  \gamma_{2i+1}(\nu(G))  \leq \gamma_{i+s+1}(\nu(G)).
\end{align*}

Furthermore each element $[x_{jk}, y_{jk}^{\varphi}]^p\in \gamma_{i+s+1}(\nu(G))$. Indeed by Proposition~\ref{thm.Hall} we have
\[
[x_{jk},y_{jk}^{\varphi}]^p \equiv [x_{jk}^p,y_{jk}^{\varphi}] \pmod{\gamma_2(K)^p \gamma_p(K)},
\]
where $K=\langle x_{jk}, [x_{jk},y_{jk}^{\varphi}]\rangle$. Observe that 
\begin{align*}
&[x_{jk}^p, y_{jk}^{\varphi}] \in [\gamma_i(G)^p, G^{\varphi}] \leq \gamma_{i+s+1}(\nu(G));\\[2mm]
&\gamma_2(K)^p \gamma_p(K) \leq \gamma_2(K) \leq [\gamma_i(G), G^{\varphi}, \gamma_i(G)] \leq  \gamma_{2i+1}(\nu(G))  \leq \gamma_{i+s+1}(\nu(G)). 
\end{align*}
This means that each $\alpha_j^p\in \gamma_{i+s+1}(\nu(G))$, so $\alpha \in \gamma_{i+s+1}(\nu(G))$. Therefore $[\gamma_i(G), G^{\varphi}]^p \leq \gamma_{i+s+1}(\nu(G))$. This concludes the proof. 
\end{proof}

We are now in a position to complete the proof of Theorem \ref{thmA}.

\begin{proof}[Proof of Theorem \ref{thmA}]
\noindent (a) follows directly from Proposition~\ref{prop.conditions}. Therefore, we need to prove (b) and (c). 

First notice that if $p$ is odd, then $\gamma_{m}(G)$ is a powerful $p$-group. Indeed, by hypothesis we have $$[\gamma_{m}(G),\gamma_{m}(G)]\leq \gamma_{2m}(G)\leq \gamma_{m+s}(G)=\gamma_{m}(G)^p.$$ 

Therefore, we can assume that $\gamma_{m}(G)$ is a powerful $p$-group  for every $p$, and we prove (b) and (c) simultaneously. Now,  by Lemma \ref{power} we have $\Pi_j(\gamma_{m}(G)))=\gamma_{m}(G)^{p^{j}}$ for all $j\geq 1$ and for every $p$. Let $p^t$ be the exponent of $\gamma_m(G)$. Thus
\[
\gamma_{m+ts}(G)=~\gamma_m(G)^{p^t}=~1.
\]
Therefore, from Proposition \ref{gammanu} we obtain that $\gamma_{m+ts+1}(\nu(G))=~1$. On the other hand, by item (a), we have $\gamma_{i+1}(\nu(G))^p\leq \gamma_{i+s+1}(\nu(G))$ for $i\geq m$. Therefore $$\gamma_{m+1}(\nu(G))^{p^t}\leq\Pi_t(\gamma_{m+1}(\nu(G))) \leq \gamma_{m+ts+1}(\nu(G))=1,$$ and the proof is concluded.
\end{proof}


\section{The exponent of $\nu(G)$}

Throughout the sequel $N$ denotes a normal subgroup of a finite $p$-group $G$. For the sake of brevity, we write $K = NN^{\varphi}[N,G^{\varphi}][G,N^{\varphi}]$ instead of $\ker(\widetilde{\pi})$ (see Lemma \ref{lem.general}, above). 

\begin{prop}\label{propN}
\ \ 
\begin{itemize}
\item[(a)] $\gamma_s(K)= \gamma_s(N)\gamma_s(N^{\varphi})[\gamma_{s-1}(N),N^{\varphi}][N,\gamma_{s-1}(N^{\varphi})]$ for $s\geq 2$.
    \item[(b)] If $p\geq 3$, $n\in \mathbb{N}$ such that $1<n<p$ and $\gamma_n(N)\leqslant N^p$, then  $\gamma_{n+1}(K)\leq \gamma_2(N)^p\gamma_2(N^{\varphi})^p[N,N^{\varphi}]^p$;  
    \item[(c)] If $p=2$ and $N$ is powerful, then $\gamma_3(K)\leq \gamma_2(N)^4\gamma_2(N^{\varphi})^4[N,N^{\varphi}]^4$.
\end{itemize}
\end{prop}

\begin{proof}
(a) Clearly  $$ \gamma_s(N)\gamma_s(N^{\varphi})[\gamma_{s-1}(N),N^{\varphi}][N,\gamma_{s-1}(N^{\varphi})] \leq \gamma_s(K),$$ for every $s\geq 2$. To prove the other inclusion we argue by induction on~ $s$.

Let $X=\{n_1,n_2^{\varphi},[n_3,g_1^{\varphi}],[g_2,n_4^{\varphi}] \ | \ n_i\in N, g_j\in G\}$ be a set of generators of $K$. Assume that $s=2$. It suffices to show that each commutator of weight 2 in the generators belongs to $\gamma_2(N)\gamma_2(N^{\varphi})[N,N^{\varphi}]$ since it is a normal subgroup of $\nu(G)$. 

Let $n_1, n_2\in N$, $g_1, g_2 \in G$ and $n = [n_1,g_1], n'=[g_2,n_2] \in N$. Then 
\begin{align*}
&[n_1, n_2^{\varphi}] \in [N,N^{\varphi}];\\[2mm]
&[n_1, g_1^{\varphi},n_2] =[n_1, g_1,n_2^{\varphi}]=[n^{-1},n_2^{\varphi}]\in [N, N^{\varphi }];\\[2mm]
&[g_2,n_2^{\varphi},n_1]=[g_2,n_2,n_1^{\varphi}]=[n',n_2^{\varphi}] \in [N, N^{\varphi }];\\[2mm]
&[[g_2,n_2^{\varphi}],[n_1, g_1^{\varphi}]]=[n',n^{\varphi}]\in [N, N^{\varphi }].
\end{align*}

Of course, $[n_1,n_2]\in \gamma_2(N)$ and $[n_1^{\varphi},n_2^{\varphi}]\in \gamma_2(N^{\varphi})$. Then for $s=2$ the inclusion holds.
Now assume $s \geq 2$ and suppose by induction hypothesis that $\gamma_s(K)= \gamma_s(N)\gamma_s(N^{\varphi})[\gamma_{s-1}(N),N^{\varphi}][N,\gamma_{s-1}(N^{\varphi})]$. In particular $Y=\{x, x^{\varphi}, [y_1, n_1^{\varphi}], [n_2, y_2^{\varphi}] \ | \ x \in \gamma_s(N), y_i \in \gamma_{s-1}(N), n_i \in N\}$ is a set of generators for $\gamma_s(K)$. Therefore we need to show that $[\alpha, \beta] \in \gamma_{s+1}(N)\gamma_{s+1}(N^{\varphi})[\gamma_{s}(N),N^{\varphi}][N,\gamma_{s}(N^{\varphi})]$ for every $\alpha \in X$ and $\beta \in Y$.
Let $x\in \gamma_s(N)$, $y\in \gamma_{s-1}(N)$,  $m,n,n'\in N$ and $g_1, g_2 \in G$ and $n = [n_1,g_1], n'=[g_2,n_2] \in N$. Then

\begin{align*}
&[x, m] \in \gamma_{s+1}(N); \hspace{0.5cm}[x,m^{\varphi}]\in [\gamma_{s}(N),N^{\varphi}];\\[2mm]
&[x^{\varphi}, m] \in [\gamma_s(N^\varphi),N];\hspace{0.5cm}[x^{\varphi},m^{\varphi}]\in \gamma_{s+1}(N^{\varphi});\\[2mm]
&[n_1, g_1^{\varphi},x] =[n_1, g_1, x^\varphi]=[n,x^{\varphi}]\in [N, \gamma_s(N^{\varphi })];\\[2mm]
&[g_2,n_2^{\varphi},x]=[g_2, n_2, x^\varphi]=[n',x^{\varphi}] \in [N, \gamma_s(N^{\varphi })];
\\[2mm]
&[[y, m^{\varphi}],n_1]=[[y, m], n_1^\varphi] \in [\gamma_s(N), N^\varphi];
\\[2mm]
&[[y, m^{\varphi}],[n_1,g_1^{\varphi}]]=[[y, m], n^{\varphi}]\in [\gamma_s(N), N^{\varphi }];\\[2mm]
&[[y, m^{\varphi}],[g_2, n_2^{\varphi}]]=[[y,m],(n')^{\varphi}]\in [\gamma_s(N), N^{\varphi }];\\[2mm]
&[[m, y^{\varphi}],n_1]=[[m, y], n_1^\varphi] \in [\gamma_s(N), N^\varphi];\\[2mm]
&[[m, y^{\varphi}],[n_1,g_1^{\varphi}]]=[[m,y],n^{\varphi}] \in [\gamma_s(N), N^\varphi];\\[2mm]
&[[m, y^{\varphi}],[g_2, n_2^{\varphi}]]=[[m,y],(n')^{\varphi}] \in [\gamma_s(N), N^\varphi].
\end{align*}
This suffices to conclude that $\gamma_{s+1}(K)\leq \gamma_{s+1}(N)\gamma_{s+1}(N^{\varphi})[\gamma_{s}(N),N^{\varphi}][N,\gamma_{s}(N^{\varphi})]$, and we are done.\\ 

\noindent (b) Consider $p\geq3$ and $1<n<p$, with $n\in \mathbb{N}$. By the previous item we have $\gamma_{n+1}(K) = \gamma_{n+1}(N) \gamma_{n+1}(N^\varphi) [\gamma_n(N), N^{\varphi}] [N, \gamma_n(N^\varphi)]$. Since $\gamma_n(N)\leq N^p$ follows that $$\gamma_{n+1}(K)\leq [N^p, N] [(N^{\varphi})^p, N^{\varphi}] [N^p, N^{\varphi}] [N, (N^{\varphi})^p].$$

First we will prove that $[N^p, N]\leq \gamma_2(N)^p$. Since $n<p$, by Lemma \ref{lem.hallformula} we have 
$$\begin{array}{ccl}
     [N^p, N] & \leq &  [N, N]^p [N,\ _p\ N] = \gamma_2(N)^p[\gamma_{p-1}(N), N, N] \\
     & \leq & \gamma_2(N)^p[\gamma_n(N), N, N] \leq \gamma_2(N)^p[N^p, N, N].
\end{array}$$

Applying Lemma~\ref{normalinc} to $[N^p, N]$ and $[N, N]^p$, we deduce that $[N^p, N] \leq \gamma_2(N)^p$. Clearly, in the same way we have $[(N^{\varphi})^p, N^{\varphi}] \leq \gamma_{2}(N^{\varphi})^p$.

Now, it remains to prove that $[N^p, N^{\varphi}] \leq [N, N^{\varphi}]^p$. For, let $x,y\in N$. By Theorem~\ref{thm.Hall}, 
$$[x^p, y^\varphi] \equiv\ [x, y^\varphi]^p (\text{mod}\ \gamma_2(L)^p\gamma_p(L)),$$
where $L=\langle x, [x, y^\varphi]\rangle$. Note that  $\gamma_2(L)^p\leq [N, N^{\varphi}, N]^p=[N, N, N^{\varphi}]^p\leq[N, N^{\varphi}]^p$ and 
$$\begin{array}{ccl}
    \gamma_p(L) & \leq & [N, N^{\varphi}, N,\ _{p-2}\ N]=[\gamma_{p-1}(N), N, N^{\varphi}]\\
     & \leq & [N^p, N, N^{\varphi}]=[N^p, N^{\varphi}, N^{\varphi}]\leq [N^p, N^{\varphi}, \nu(G)]
\end{array}$$
Considering all the elements $x, y\in N$, we deduce that
$$[N^p, N^{\varphi}]\leq [N, N^{\varphi}]^p[N^p, N^{\varphi}, \nu(G)].$$ 

Note that $[N^p, N^{\varphi}]$ and $[N, N^{\varphi}]^p$ are normal subgroups of $\nu(G)$. So, applying Lemma \ref{normalinc} to these normal subgroups we get $[N^p, N^{\varphi}] \leq [N, N^{\varphi}]^p$. Similarly we obtain $[N, (N^{\varphi})^p]\leq  [N, N^{\varphi}]^p$.

Therefore $$\gamma_{n+1}(N) \gamma_{n+1}(N^{\varphi}) [N^p, N^{\varphi}] [N,(N^p)^{\varphi})] \leq \gamma_2(N)^p\gamma_2(N^{\varphi})^p[N,N^{\varphi}]^p$$ and the proof is complete for $p\geq3$. \\

\noindent (c) Now consider $p=2$. Since $N$ is a potent $p$-group, by Lemma~\ref{lem.lower.potent} we have $\gamma_3(N) \leq \gamma_2(N)^4$ so $$\gamma_3(K)=\gamma_{3}(N)\gamma_{3}(N^{\varphi})[\gamma_2(N),N^{\varphi}][N,\gamma_2(N^{\varphi})]\leq \gamma_{2}(N)^4\gamma_{2}(N^{\varphi})^4[N^4,N^{\varphi}][N,(N^{\varphi})^4]$$

We need to prove that $[N^4, N^{\varphi}]\leq  [N, N^{\varphi}]^4$. Let $n, m\in N$, by Theorem~\ref{thm.Hall}, 
$$[n^4, m^\varphi] \equiv\ [n, m^\varphi]^4 (\text{mod}\ \gamma_2(L)^4\gamma_2(L)^2\gamma_4(L)),$$
where $L=\langle n, [n, m^\varphi]\rangle$. Note that  $\gamma_2(L)\leq [N, N^{\varphi}, N]=[N, N, N^{\varphi}]$, this implies that  
\begin{align*}
&\gamma_2(L)^{4} \leq \gamma_2(L)^2\leq[N, N, N^{\varphi}]^2\leq[N^4, N^{\varphi}]^2\\[2mm]
&\gamma_4(L) \leq [N, N, N^{\varphi}, L, L] \leq [N, N, N^{\varphi}, \nu(G), \nu(G)] \leq [N^4, N^{\varphi}, \nu(G)].
\end{align*}

Therefore $[n^4, m^{\varphi}] \in [N, N^{\varphi}]^4[N^4, N^{\varphi}, \nu(G)][N^4, N^{\varphi}]^2$. 
By commutators relations we can conclude that for each element $\alpha\in N^4$ and each $m\in N$  we have $[\alpha, m^{\varphi}] \in [N, N^{\varphi}]^4[N^4, N^{\varphi}, \nu(G)][N^4, N^{\varphi}]^2$, that is, 
$$[N^4, N^{\varphi}]\leq [N, N^{\varphi}]^4[N^4, N^{\varphi}, \nu(G)][N^4, N^{\varphi}]^2.$$ 

Since the subgroups $[N^4, N^{\varphi}]$ and $[N, N^{\varphi}]^4$ are normal in $\nu(G)$, we can apply Lemma~\ref{normalinc}, obtaining $[N^4, N^{\varphi}] \leq [N, N^{\varphi}]^4$. In the same way $[N, (N^{\varphi})^4]\leq  [N, N^{\varphi}]^4$.

Therefore $$\gamma_{3}(N)\gamma_{3}(N^{\varphi})[N^4,N^{\varphi}][N,(N^{\varphi})^4]\leq \gamma_2(N)^4\gamma_2(N^{\varphi})^4[N,N^{\varphi}]^4$$ and the proof is complete.
\end{proof}

\begin{cor}\label{s-lowercentral}
If $N$ is potent and $s \geq 2$, then the $s$-th term of the lower central series $\gamma_s(K)$ is potently embedded in $K$. 
\end{cor}
\begin{proof}
The proof is by induction on $s$. If $s=2$, by Proposition~\ref{propN} and by definition we have

$$\begin{array}{lcl}
[\gamma_2(K), _{p-2}K]=\gamma_p(K) & \leq &  \gamma_2(N)^p\gamma_2(N^{\varphi})^p[N,N^{\varphi}]^p\\
\ & \leq & (\gamma_2(N)\gamma_2(N^{\varphi})[N,N^{\varphi}])^p = \gamma_2(K)^p,\ \text{for}\ p\geq 3,
\end{array}$$

$$\begin{array}{lcl}
[\gamma_2(K), K]=\gamma_3(K) & \leq & \gamma_2(N)^4\gamma_2(N^{\varphi})^4[N,N^{\varphi}]^4\\
\ & \leq & (\gamma_2(N)\gamma_2(N^{\varphi})[N,N^{\varphi}])^4 = \gamma_2(K)^4, \ \text{for}\ p=2.  
\end{array}$$

This means that $\gamma_2(K)$ is potently embedded in $K$.

Suppose by induction hypothesis that $[\gamma_{s}(K), _{p-2}K] \leq \gamma_{s}(K)^p$, if $p\geq 3$, and $[\gamma_{s}(K), K] \leq \gamma_{s}(K)^4$, if $p=2$. Now Lemma \ref{lem.hallformula} yields  
$$\begin{array}{lcl}
[\gamma_{s+1}(K), _{p-2}K]& \leq & [\gamma_s(K)^p, K]\\
	    \ & \leq & [\gamma_s(K), K]^p[K,\  _p\ \gamma_s(K)]\\
	    \ & \leq & \gamma_{s+1}(K)^p[\gamma_{s+1}(K),_{p-2} K, K], \ \text{for} \  p\geq 3,
\end{array}$$

$$\begin{array}{lcl}
[\gamma_{s+1}(K), K]& \leq & [\gamma_s(K)^4, K]\\
	    \ & \leq & [\gamma_s(K), K]^4[K,\  _2\ \gamma_s(K)]^2[K, _4\gamma_s(K)]\\
	    \ & \leq & \gamma_{s+1}(K)^4[\gamma_{s+1}(K),  K]^2 [\gamma_{s+1}(K), K, K, K], \ \text{for} \ p=2.
\end{array}$$

Therefore, by Lemma \ref{normalinc} $[\gamma_{s+1}(K), _{p-2}K]\leq \gamma_{s+1}(K)^p$, if $p\geq 3$ and $[\gamma_{s+1}(K),K]\leq\gamma_{s+1}(K)^4$, if $p=2$, as desired.
\end{proof}

\begin{cor}\label{corN}
If $N$ is potent or $\gamma_p(N)=1$, then $\exp(K)$ divides $\p \cdot \exp(N)$. 
\end{cor}

\begin{proof}
Assume that $\exp(N)=p^e$. As $K$ is generated by $N^{G^{\varphi}}$ and  $(N^{\varphi})^G$, it follows that $K=\Omega_e(K)$.

Now let $p\geq 3$. On the one hand, if $N$ is potent, according to Corollary~\ref{s-lowercentral}, we conclude that $$\gamma_{2(p-1)}(K) = [\gamma_p(K),_{p-2}K] \leq \gamma_{p}(K)^p.$$ 

On the other hand, if $\gamma_p(N)=1$, by Proposition~\ref{propN} (a) we have $\gamma_{p+1}(K)=1$. So $$\gamma_{2(p-1)}(K) \leq \gamma_{p+1}(K)=1 \leq \gamma_p(K)^p.$$

In both cases we can apply Lemma~\ref{lem.exponent} with $k=2, r=p$ and $s=1$, obtaining $\exp(K) = \exp(\Omega_e(K))$ divides $p^{e+1}$. 

Analogously, for $p=2$ we have $\gamma_{3}(K) \leq \gamma_{2}(K)^4$, if $N$ is potent, and $\gamma_3(K)=1 \leq \gamma_2(K)^4$, if $\gamma_2(N)=1$. Then, applying Lemma~\ref{lem.exponent} with $k=3, r=2$ and $s=2$, we obtain that $\exp(K)=\exp(\Omega_e(K))$ divides $2^{e+2}$.
\end{proof}

The proof of Theorem \ref{thm.potent} is now easy. 

\begin{proof}[Proof of Theorem~\ref{thm.potent}]
(a) According to Lemma \ref{lem.general}~(b), we deduce that $\exp(\nu(G))$ divides $\exp(\nu(G/N)) \cdot \exp(K)$. 
By Corollary~\ref{corN}, $\exp(K)$ 
divides $\p\cdot \exp(N)$, which completes the proof. \\

\noindent (b) Arguing as in the previous paragraph, it is sufficient to prove that $\exp(K)$ divides $\exp(N)$. So, let $p\geq 5$ and suppose that $\gamma_{p-2}(N) \leq N^p$. Applying Proposition~\ref{propN} (b), for $n=p-2$ we have that $\gamma_{p-1}(K) \leq \gamma_2(K)^p$. In particular, $K$ is potent. By Lemma~\ref{lem.exponent} with $k=1$, $r=2$ and $s=1$ we obtain that $\exp(K)=\exp(\Omega_e(K))$ divides $\exp(N)$, since $K=\ker(\tilde{\pi}) = \Omega_e(K)$.
\end{proof}

Arguing as in the proof of the above result and using Lemma~\ref{lem.general} (c) and the fact that $\exp([N, G^{\varphi}][G, N^{\varphi}])$ divides $\exp(K)$, which in turn divides $\exp(N)$, we obtain the following:

\begin{cor}\label{cor.potent}
Let $N$ be a normal subgroup of a $p$-group $G$. 
\begin{itemize}
    \item[(a)] If $N$ is potent or $\gamma_{p}(N)=1$, then $\exp([G,G^{\varphi}])$ divides $\p\cdot \exp([G/N,(G/N)^{\varphi}])\cdot\exp(N)$.
    \item[(b)] If $\gamma_{p-2}(N) \leq N^p$,  then $\exp([G,G^{\varphi}])$ divides $\exp([G/N,(G/N)^{\varphi}])\cdot\exp(N)$. 
\end{itemize}
\end{cor}

A finite $p$-group $G$ is called regular if $x^py^p \equiv~(xy)^p \mod  H^p$, for every $x,y \in G$ and $H = H(x,y) = \langle x,y\rangle'$. It is well-known that if $G$ is a regular $p$-group and $G$ is generated by a set $X$, then $\exp(G) = \max\{|x| \mid x \in X\}$, where $|x|$ denotes the order of the element $x$ in the group $G$ (see \cite[1.2.13~(i)]{LM}).

\begin{proof}[Proof of Theorem \ref{corlog}]
Recall that $G$ is a $p$-group of nilpotency class $c$ and $n=\lceil \log_{p}(c+1)\rceil$. We prove by induction on $c$ that $\exp([G,G^{\varphi}])$ divides $\exp{(G)}^n$.

For any $c$ let $N=\gamma_{j}(G)$ and $H=[N,G^{\varphi}]\cdot [G,N^{\varphi}]$ where $j=\lceil \frac{c+1}{p}\rceil$. Then $H\leq \gamma_{j+1}(\nu(G))$ and hence $\gamma_p(H)\leq \gamma_{c+p+1}(\nu(G))\leq \gamma_{c+2}(\nu(G))=1$. In particular, $H$ is a regular $p$-group. Let $x\in N$ and $y^{\varphi}\in G^{\varphi}$. We will prove that if $\exp(N)=p^e$, then $H^{p^e}=1$. Applying Theorem \ref{thm.Hall} we have 
\[
[x^{p^e}, y^{\varphi}] \equiv [x,y^{\varphi}]^{p^e}  \pmod {\gamma_{2}(L)^{p^e}\gamma_{p}(L)^{p^{e-1}}\gamma_{p^2}(L)^{p^{e-2}}\ldots \gamma_{p^e}(L)},
\]
where $L= \langle x, [x,y^{\varphi}]\rangle$. On the one hand, $\gamma_p(L) \leq \gamma_{pj+1}(\nu(G)) \leq \gamma_{c+2}(\nu(G))=1$. If $p=2$, then $\gamma_2(L)=1$ and we obtain that $[x,y^{\varphi}]^{p^e}=1$. In particular, $H^{p^e}=1$

If $p$ is odd, then $[x,y^{\varphi}]^{p^e}  \in \gamma_{2}(L)^{p^e}$.  Since $[x,y^{\varphi},x]=[x,y,x^{\varphi}]\in [N,N^{\varphi}]$ we conclude that $[x,y^{\varphi}]^{p^e} \in \gamma_{2}(L)^{p^e} \leq [N,N^{\varphi}]^{p^e}$. Therefore, it is sufficient to prove that $[N,N^{\varphi}]^{p^e}=1$. 

Repeating the process with $a, b\in N$ we obtain that $[a,b^{\varphi}]^{p^e}\in [N',N^{\varphi}]^{p^e}$, so $H^{p^e}\leq [N,N^{\varphi}]^{p^e}\leq [N',N^{\varphi}]^{p^e} $. It is clear that using this process $p$-times eventually we obtain that $[N,N^{\varphi}]^{p^e}=1$, since $\gamma_p(N)=1$. Note that if $c \leq p-1$, then $N=G$ since $j=1$. In particular, $H=[G,G^{\varphi}]$ and $H^{\exp(G)}=1$. Thus, it remains to prove the case when $c\geq p$.

Now, Lemma \ref{lem.general} (c) implies that $\exp([G,G^{\varphi}])$ divides $$\exp([G/N,G^{\varphi}/N^{\varphi}]) \cdot \exp(H).$$
As $G/N$ has nilpotency class $\left\lceil\frac{c+1}{p}\right\rceil - 1$, by induction we obtain that $\exp([G/N,G^{\varphi}/N^{\varphi}])$ divides $\exp(G)^m$ where $m=\left\lceil \log_{p} \left\lceil\frac{c+1}{p}\right\rceil\right\rceil=\left\lceil \log_{p}(\frac{c+1}{p})\right\rceil=\left\lceil\log_{p}(c+1)\right\rceil-~1$. Therefore, $\exp([G,G^{\varphi}])$ divides $\exp(G)^m\cdot \exp(G)$ and 
the result follows.
\end{proof}

The previous result improve the bounds obtained \cite{APT,Moravec.Schur,Sambonet17} (cf. \cite[Theorem 6.5]{APT},   \cite[Section 2]{Moravec.Schur} and \cite[Theorem 1.1]{Sambonet17}).

\section{Applications: finite $p$-groups of fixed coclass}

Recall that the coclass of a $p$-group of order $p^n$ and nilpotency class $c$ is defined to be $r=n-c$. Finite $p$-groups of coclass 1 are also known as $p$-groups of maximal class. For $p=2$ these are known to be either dihedral, semidihedral or quaternion groups \cite[Corollary 3.3.4 (iii)]{LM}. 

Let $G$ be a $p$-group of maximal class of order greater than or equal to $p^4$. We define $G_1=C_G(\gamma_2(G)/\gamma_4(G))$. In other words $G_1$ consists of the elements $x\in G$ such that $[x,\gamma_2(G)]\leq \gamma_4(G)$.
It is well-known that the subgroup $G_1$ is a characteristic maximal subgroup of $G$. Another structural property about the subgroup $G_1$ is given in the next result. 

\begin{thm}\cite[Corollary 3.3.6]{LM}
If $G$ is a $p$-group of maximal class of order greater than or equal to $p^{p+2}$, then $\gamma_p(G)=G_1^p$.
\end{thm}

More information on $p$-groups of maximal class can be found in \cite[Chapter 3]{LM}. We are now in a position to prove Corollary~\ref{cor.maximal}.

\begin{proof}[Proof of Corollary \ref{cor.maximal}]
First of all,  we prove that $G$ has a potent maximal subgroup or a maximal subgroup of class at most $p-1$. If $p=2$, then $G$ has a cyclic maximal subgroup. Thus we can assume that $p$ is odd.

On the one hand, if $|G|\leq p^{p+1}$, then its maximal subgroups have order at most $p^p$ and then nilpotency class at most $p-1$.

On the other hand, assume that $|G|\geq p^{p+2}$. In this case $[G_1,G_1]=[G_1,\gamma_2(G)]$, since $|G_1:\gamma_2(G)|=p$. Thus, as $[G_1,G_1]\leq \gamma_4(G)$, it follows that $$\gamma_{p-1}(G_1)=[[G_1,G_1],\ _{p-3}\ G_1]\leq [\gamma_4(G),\ _{p-3}\ G_1]\leq \gamma_p(G)=G_1^p,$$ and $G_1$ is a potent maximal subgroup.

Now, by Theorem~\ref{thm.potent} and Corollary~\ref{corN} we obtain that $\exp(\nu(G))$ divides $\p\cdot\exp(\nu(C_p))\cdot \exp(G)=\p^2\cdot \exp(G)$, since $\exp(\nu(C_p))=4$ if $p=2$ or $p$ if $p$ is odd.
\end{proof}

The following result is an immediate consequence of Corollary \ref{cor.maximal}. 

\begin{cor}
Let $p$ be a prime and $G$ a $p$-group of maximal class. Then $\exp(\mu(G))$ and $\exp([G,G^{\varphi}])$ divide $\p^2 \cdot \exp(G)$.
\end{cor}

Let $p$ be a prime and $r$ an integer we define the integer $m(p,r)$ by $m(p,r)=(p-1)p^{r-1}$ for $p$ odd and $m(2,r)=2^{r+2}$. It is well-known that if $G$ is a $p$-group of coclass $r$, then $\gamma_{m(p,r)}(G)$ is powerful, see for instance \cite[Theorem 6.3.1 and 6.3.2]{LM}. Recall that $d(\gamma_m(G))$ is the minimal cardinality for a generating set of $\gamma_m(G)$. Moreover, we have the following two results.

\begin{thm}\cite[Theorem 6.3.8]{LM}\label{coclass2}
Let $G$ be a finite $2$-group of coclass $r$ and nilpotency class $c$ and let $m=m(2,r)$ and $s=d(\gamma_m(G))$. If $c\geq 2^{r+3}$, then the following hold:
\begin{itemize}
    \item[(a)] $\gamma_i(G)^2=\gamma_{i+s}(G)$ for all $i\geq m$;
    \item[(b)] $s=2^d$ with $0\leq d\leq r+1$.
\end{itemize}
\end{thm}

\begin{thm}\cite[Theorem 6.3.9]{LM}\label{coclassp}
Let $G$ be a finite $p$-group of coclass $r$ and nilpotency class $c$ for $p$ odd, and let $m=m(p,r)$ and $s=d(\gamma_m(G))$. If $c\geq 2p^{r}$, then the following hold:
\begin{itemize}
    \item[(a)] $\gamma_i(G)^p=\gamma_{i+s}(G)$ for all $i\geq m$;
    \item[(b)] $s=(p-1)p^d$ with $0\leq d\leq r-1$.
\end{itemize}
\end{thm}

We are now in a position to prove Theorem \ref{thm.cc}. 

\begin{proof}[Proof of Theorem~\ref{thm.cc}]
Let $m=m(p,r)$ and consider the quotient group $\bar{G}=G/\gamma_{m+1}(G)$. By Theorem \ref{corlog} we have that $\exp ([\bar{G},\bar{G}^{\varphi}])$ divides $\exp(\bar{G})^n$ where $n=\lceil \log_{p}(m+1)\rceil$. If $p$ is odd, then  $n\leq \lceil \log_p(p^r)\rceil = r$ as $m+1 \leq p^r$. If $p=2$, then $n=r+3$. By Lemma \ref{lem.general} the kernel of the canonical epimorphism $\widetilde{\pi}: \nu(G) \to \nu(\overline{G})$ is the subgroup $\gamma_{m+1}(G)\gamma_{m+1}(G)^{\varphi}[\gamma_{m+1}(G),G^{\varphi}][G,\gamma_{m+1}(G)^{\varphi}]$ which is contained in $\gamma_{m+1}(\nu(G))$ by Proposition~\ref{gammanu}. Now, applying Theorem \ref{thmA}  we have $\exp(\gamma_{m+1}(\nu(G)))\leq \exp(\gamma_{m}(G))$.
\end{proof}

\section*{Acknowledgements}
The work of the first and the second authors were supported by DPI/UnB and FAPDF-Brazil. The third author was supported by CNPq-Brazil. The last author was supported by the ``National Group for Algebraic and Geometric Structures, and their Applications" (GNSAGA - INdAM). The last author is also grateful to the Department of Mathematics of the University of Brasilia for its hospitality and support, while this investigation was carried out. 
Finally, the authors are very grateful to the referees who have carefully read the manuscripts and pointed out several mistakes and typographical errors. Moreover, their insightful comments were valuable for the improvement of this new version.  

\bibliographystyle{plain}

\begin{thebibliography}{10}

\bibitem{APT} A.\,E. Antony, P. Komma and V.\,Z. Thomas, {\it Commutator expansions and the  Schur Multiplier}, preprint available at ArXiV:1906.09585v3. 
    
\bibitem{BdMGN} R. Bastos, E. de Melo, N. Gon\c calves and R. Nunes, {\it Non-abelian tensor square and related constructions of $p$-groups}, Arch. Math., {\bf 114} (2020) pp. 481--490. 

\bibitem{BFM} R.\,D. Blyth, F. Fumagalli and M. Morigi, {\it Some structural results on the non-abelian tensor square of groups}, J. Group Theory, {\bf 13} (2010) pp. 83--94.

\bibitem{BL} R. Brown, and J.-L. Loday, {\it Van Kampen theorems for diagrams of spaces}, Topology, {\bf 26} (1987) pp. 311--335.

\bibitem{BJR} R. Brown, D.\,L. Johnson and E.\,F. Robertson, {\it Some computations of non-abelian tensor products of groups}, J. Algebra, {\bf 111} (1987)  pp. 177--202.
    
\bibitem{BuenoRocco} T.\,P. Bueno and N.\,R. Rocco, {\it On the $q$-tensor square of a group}, J. Group Theory, {\bf 14} (2011) pp. 785--805. 

\bibitem{Dennis} R.\,K. Dennis, {\it In search of new ``homology'' functors having a close relationship to $K$-theory}, Preprint, Cornell University, Ithaca, NY, 1976.   
    
\bibitem{D} J.\,D. Dixon, M.\,P.\,F. du Sautoy, A. Mann and D. Segal, {\em Analytic $p$-adic groups.} Cambridge 1991.
    
\bibitem{BN08} B. Eick and W. Nickel, {\it Computing the Schur multiplicator and the nonabelian tensor square of a polycyclic group}, J. Algebra, {\bf 320} (2008) pp. 927--944.

\bibitem{Ellis} G. Ellis, {\it On the tensor square of a prime power group}, Arch. Math., {\bf 66} (1996) pp. 467--469.   
    
\bibitem{EL} G. Ellis and F. Leonard, {\it Computing Schur multipliers and tensor products of finite groups},  Proc. Royal Irish Acad., {\bf 95A} (1995) pp. 137--147.

\bibitem{JJJ} G.\,A. Fern\'andez-Alcober, J. Gonz\'alez-S\'anches and A. Jaikin-Zapirain, \textit{Omega subgroups of pro-$p$ groups}, Isr. J. Math., {\bf 166} (2008) pp. 393--412.

   
\bibitem{JJ} J. Gonz\'alez-S\'anches and A. Jaikin-Zapirain, \textit{On the structure of normal subgroups of potent $p$-groups}, J. Algebra, {\bf 276} (2004) pp. 193--209.
    
\bibitem{Khukhro} E.\,I. Khukhro, \textit{$p$-Automorphisms of finite $p$-groups}, Cambridge University Press, 1998.    
    
\bibitem{LM} C. R. Leedham-Green and S. McKay, \textit{The Structure of Groups of Prime Power Order}, Oxford University Press, 2002. 
    
\bibitem{Miller} C. Miller, {\it The second homology group of a group: relations among commutators}, Proc. Am. Math. Soc., {\bf 3} (1952) pp. 588--595.

\bibitem{Moravec.Schur} P. Moravec, {\it Schur multiplier and power endomorphisms of groups}, J. Algebra {\bf 308}  (2007) pp. 12--25.
    
\bibitem{M09} P. Moravec, {\it Groups of prime power order and their nonabelian tensor squares}, Isr. J. Math., {\bf 174} (2009) pp. 19--28.
    
\bibitem{Moravec.cc}  P. Moravec, {\it On the Schur multipliers of finite $p$-groups of given coclass}, Isr. J. Math., {\bf 185} (2011) pp. 189--205.   

\bibitem{NR1} N.\,R. Rocco, {\it On a construction related to the non-abelian tensor square of a group}, Bol. Soc. Brasil Mat., {\bf 22} (1991) pp. 63--79.

\bibitem{NR2} N.\,R. Rocco, {\it A presentation for a crossed embedding of finite solvable groups}, Comm. Algebra, {\bf 22} (1994) pp. 1975--1998.
    
\bibitem{Sambonet} N. Sambonet, {\it The unitary cover of a finite group and the exponent of the Schur multiplier}, J. Algebra, {\bf 426} (2015) pp. 344--364.

\bibitem{Sambonet17} N. Sambonet, {\it Bounds for the exponent of the Schur multiplier}, J. Pure Appl. Algebra {\bf 221} (2017) 2053--2063. 
\end{thebibliography}

\end{document}